\documentclass[12pt,twoside]{amsart}
\usepackage{amsmath, amsthm, amscd, amsfonts, amssymb, graphicx}
\usepackage{enumerate}
\usepackage[colorlinks=true,
linkcolor=blue,
urlcolor=cyan,
citecolor=red]{hyperref}
\usepackage{mathrsfs}
\newtheorem{theorem}{Theorem}

\newtheorem{lemma}{Lemma}[section]
\newtheorem{remark}{Remark}[section]

\newtheorem{corollary}{Corollary}[section]

\newtheorem{proposition}{Proposition}[section]
\numberwithin{equation}{section}

\begin{document}
\title{ New sharp inequalities for operator means  }
\author{Shigeru Furuichi, Hamid Reza Moradi  and Mohammad Sababheh }
 \maketitle

\begin{abstract}
New sharp multiplicative reverses of the operator  means inequalities  are presented, with a simple discussion of squaring an operator inequality. As a direct consequence, we extend the operator P\'olya-Szeg\"o inequality to arbitrary operator means.  Furthermore, we obtain some new lower and upper bounds for the Tsallis relative operator entropy, operator monotone functions and positive linear maps. 
\end{abstract}
\pagestyle{myheadings}
\markboth{\centerline {S. Furuichi, H.R. Moradi \& M. Sababheh}}
{\centerline {}}
\bigskip
\bigskip
\section{Notation and preliminaries}
Let $\mathcal{B}\left( \mathcal{H} \right)$ be the ${{C}^{*}}$-algebra of all bounded linear operators acting on a complex (separable) Hilbert space $\left( \mathcal{H},\left\langle \cdot,\cdot \right\rangle  \right)$ and $I$ be its identity. An operator $A\in \mathcal{B}\left( \mathcal{H} \right)$  is said to be positive semi-definite (denoted by $A\geq  0$) if $\left\langle Ax,x \right\rangle \geq 0$ for all  vectors $x\in \mathcal{H}$. If $\left<Ax,x\right>\;>0$ for all nonzero vectors $x\in\mathcal{H}$, $ A$ is said to be positive (denotes $A>0).$ For self-adjoint operators $A,B\in \mathcal{B}\left( \mathcal{H} \right)$, $A\le B$ means $B-A$ is positive semi-definite. 

A continuous  function $f:[0,\infty)\to\mathbb{R}$ is said to be operator monotone (resp. operator monotone decreasing) if for two positive operators $A$ and $B$, the inequality $A\le B$ implies $f\left( A \right)\le f\left( B \right)$ (resp. $f\left( A \right)\ge f\left( B \right)$). It is known that a non-negative operator monotone function $f$ has the following property \cite[Theorem V.2.5]{bhatia}: for every $A,B\ge 0$, $f\left( \frac{A+B}{2} \right)\ge \frac{f\left( A \right)+f\left( B \right)}{2}$. Such a function is referred to as an operator concave function. A real function $f$, such that $-f$ is operator concave is called operator convex.
A linear map $\Phi :\mathcal{B}\left( \mathcal{H} \right)\to \mathcal{B}\left( \mathcal{H} \right)$ is positive if $\Phi \left( A \right)\ge 0$ whenever $A\ge 0$. It is said to be unital if $\Phi \left( I \right)=I$.

For positive invertible operators $A,B\in \mathcal{B}\left( \mathcal{H} \right)$, the operator weighted arithmetic, geometric and harmonic means are defined, respectively, by
\[A{{\nabla }_{v}}B=\left( 1-v \right)A+vB,\quad\text{ }A{{\sharp}_{v}}B={{A}^{\frac{1}{2}}}{{\left( {{A}^{-\frac{1}{2}}}B{{A}^{-\frac{1}{2}}} \right)}^{v}}{{A}^{\frac{1}{2}}},\quad\text{ }A{{!}_{v}}B={{\left( {{A}^{-1}}{{\nabla }_{v}}{{B}^{-1}} \right)}^{-1}}\] 
where $v\in \left[ 0,1 \right]$. When $v=\frac{1}{2}$, we drop the $v$ from the above notations. The following arithmetic-geometric-harmonic-mean  inequalities hold
\begin{equation}\label{6}
A{{!}_{v}}B\le A{{\sharp}_{v}}B\le A{{\nabla }_{v}}B.
\end{equation}
Over the years, various reverses and refinements of \eqref{6} have been obtained in the literature, e.g., \cite{3,sab_laaa,sab_mia,01}. In this paper, we will present new reverse inequalities of \eqref{6}. Our inequalities sharpen  many previously known results. See Theorem \ref{19} and its consequences below.

The operator P\'olya-Szeg\"o inequality \cite[Theorem 4]{Lee2009} is given as follows:
\begin{equation}\label{24}
\Phi \left( A \right)\sharp\Phi \left( B \right)\le \frac{M+m}{2\sqrt{Mm}}\Phi \left( A\sharp B \right),
\end{equation}
whenever $mI\le A, B\le MI$, $\Phi$ is a positive linear map on $\mathcal{B}(\mathcal{H})$ and $m,M$ are positive numbers.

Hoa et al. \cite[Theorem 2.12]{5} proved that if $\Phi $ is a positive linear map, $f$ is a nonzero operator monotone function on $\left[ 0,\infty  \right)$ and $A,B\in \mathcal{B}\left( \mathcal{H} \right)$ such that $0<mI\le A,B\le MI$, then
\begin{equation} \label{Hoa_ineq01}
f\left( \Phi \left( A \right) \right)\tau f\left( \Phi \left( B \right) \right)\le \frac{{{\left( M+m \right)}^{2}}}{4Mm}f\left( \Phi \left( A\sigma B \right) \right),
\end{equation}
where $\sigma ,\tau $ are two arbitrary operator means between the arithmetic mean $\nabla$ and the harmonic mean $!$. In this paper we extend this result to the weighted means $\tau_v,\sigma_v$ and under the sandwich assumption $sA\leq B\leq tA.$ Our results will be natural generalizations of \eqref{Hoa_ineq01}.
\section{Main Results}
In this part of the paper, we present our main results, starting with means inequalities that will be used later to obtain inequalities for operator monotone functions and positive linear maps.
\subsection{Operator means inequalities}
We begin with the following new reverse of \eqref{6}.
\begin{theorem}\label{19}
Let $A,B$ be  positive operators such that $sA\le B\le tA$ for some scalars $0<s\le t$. Then for any $v\in \left[ 0,1 \right]$,
\begin{equation}\label{12}
\frac{1}{\xi }A{{\nabla }_{v}}B\le A{{\sharp}_{v}}B\le \psi A{{!}_{v}}B,
\end{equation}
where $\xi =\max \left\{ \frac{\left( 1-v \right)+vs}{{{s}^{v}}},\frac{\left( 1-v \right)+vt}{{{t}^{v}}} \right\}$ and $\psi =\max \left\{ {{s}^{v}} \left(\left( 1-v \right)+\frac{v}{s}\right),{{t}^{v}} \left( \left( 1-v \right)+\frac{v}{t}\right) \right\}$.
\end{theorem}
\begin{proof}
Define
\[{{f}_{v}}\left( x \right)=\frac{\left( 1-v \right)+vx}{{{x}^{v}}},\quad\text{ where }0<s\le x\le t\text{ and }v\in \left[ 0,1 \right].\]
Direct calculations show that $f'(x)=v(1-v)(x-1)x^{-v-1}.$ Since $f$  is continuous on the interval $\left[ s,t \right]$, ${{f}_{v}}\left( x \right)\le \max \left\{ {{f}_{v}}\left( s \right),{{f}_{v}}\left( t \right) \right\}$. Whence,
\begin{equation}\label{7}
\left( 1-v \right)+vx\le \xi {{x}^{v}}.
\end{equation}
Now utilizing inequality \eqref{7} with $T={{A}^{-\frac{1}{2}}}B{{A}^{-\frac{1}{2}}}$ and applying a standard functional calculus argument, we obtain the first inequality in \eqref{12}.

The second one follows by applying similar arguments to the function
\[{{g}_{v}}\left( x \right)={{x}^{v}}\left( \left( 1-v \right)+\frac{v}{x} \right),\quad\text{ where }0<s\le x\le t\text{ and }v\in \left[ 0,1 \right].\]
\end{proof}

For the functions $f_v$ and $g_v$ defined in Theorem \ref{19}, we have the following inequalities that will be used later.
\begin{lemma}\label{21}
	Let $f_v$ and $g_v$  be the functions defined in the proof of Theorem \ref{19} for $x >0$. Then we have the following properties.
	\begin{itemize}
		\item[(i)] For $0< x \le 1$ and $0\le v \le \frac{1}{2}$, we have $f_v(x)\le f_v\left(\frac{1}{x}\right)$ and $g_v(x)\ge g_v\left(\frac{1}{x}\right)$.
		\item[(ii)] For $0< x \le 1$ and $\frac{1}{2}\le v \le 1$, we have $f_v(x)\ge f_v\left(\frac{1}{x}\right)$ and $g_v(x)\le g_v\left(\frac{1}{x}\right)$.
	\end{itemize}
\end{lemma}
\begin{proof}
	We set $F_v(x)\equiv f_v(x)-f_v\left(\frac{1}{x}\right)$ and $G_v(x)\equiv g_v(x)-g_v\left(\frac{1}{x}\right)$ for $x >0$. Then we calculate $\frac{dF_v(x)}{dx}=\frac{v(1-v)(x-1)(1-x^{2v-1})}{x^{v+1}}$ and $\frac{dG_v(x)}{dx}=\frac{v(1-v)(x-1)(x^{2v-1}-1)}{x^{v+1}}$.
	Then a standard calculus argument implies the four implications. 
\end{proof}
The next result is a consequence of Theorem \ref{19} with $s=\frac{m}{M}$ and $t=\frac{M}{m}$. This implies a more familiar form of means inequalities. 
\begin{corollary}\label{10}
Let $A,B$ be  positive operators such that $mI\le A,B\le MI$ for some scalars $0<m<M$. Then
\begin{equation}\label{4}
\frac{m{{\sharp}_{\lambda }}M}{m{{\nabla }_{\lambda }}M}A{{\nabla }_{v}}B\le A{{\sharp}_{v}}B\le \frac{m{{\sharp}_{\mu }}M}{m{{!}_{\mu }}M}A{{!}_{v}}B,
\end{equation}
where $\lambda =\min \left\{ v,1-v \right\}$, $\mu =\max \left\{ v,1-v \right\}$, and $v\in \left[ 0,1 \right]$.
\end{corollary}
\begin{proof}
Taking $s=\frac{m}{M}$ and $t=\frac{M}{m}$ in Theorem \ref{19}, we infer that
\[\frac{1}{\xi }A{{\nabla }_{v}}B\le A{{\sharp}_{v}}B\le \psi A{{!}_{v}}B,\]
where $\xi =\max \left\{ \frac{M{{\nabla }_{v}}m}{M{{\sharp}_{v}}m},\frac{m{{\nabla }_{v}}M}{m{{\sharp}_{v}}M} \right\}$ and $\psi =\max \left\{ \frac{M{{\sharp}_{v}}m}{M{{!}_{v}}m},\frac{m{{\sharp}_{v}}M}{m{{!}_{v}}M} \right\}$. On account of Lemma \ref{21} we get 
\[\xi =\left\{ \begin{array}{lr}
 \frac{m{{\nabla }_{v}}M}{m{{\sharp}_{v}}M}&\text{ if }v\in \left[ 0,\frac{1}{2} \right] \\ 
 \frac{M{{\nabla }_{v}}m}{M{{\sharp}_{v}}m}&\text{ if }v\in \left[ \frac{1}{2},1 \right] \\ 
\end{array} \right.\text{ }\equiv\frac{m{{\nabla }_{\lambda }}M}{m{{\sharp}_{\lambda }}M},\quad\text{ where }\lambda =\min \left\{ v,1-v \right\},\]
and
\[\psi =\left\{ \begin{array}{lr}
 \frac{M{{\sharp}_{v}}m}{M{{!}_{v}}m}&\text{ if }v\in \left[ 0,\frac{1}{2} \right] \\ 
 \frac{m{{\sharp}_{v}}M}{m{{!}_{v}}M}&\text{ if }v\in \left[ \frac{1}{2},1 \right] \\ 
\end{array} \right.\text{ }\equiv \frac{m{{\sharp}_{\mu }}M}{m{{!}_{\mu }}M},\quad\text{ where }\mu =\max \left\{ v,1-v \right\}.\]
\end{proof}

We would like to emphasize that \eqref{4} is an extension of \cite[Theorem 13]{1} to the weighted means (thanks to $\frac{M\sharp m}{M!m}= \frac{M\nabla m}{M\sharp m}$). We also remark that Corollary \ref{10} has been shown recently in \cite{submitted}.

Before proceeding further, we present the following remark about the powers of operator inequalities.

\begin{remark}
From Corollary  \ref{10}, we have the well known inequality 
\begin{equation}\label{needed_power}
\Phi\left(A\nabla B\right)\leq \frac{m\nabla M}{m\sharp M} \Phi(A\sharp B),
\end{equation}
 for a positive linear map $\Phi$ and $0<mI\le A,B\le MI$. 

It is well known that the mapping $t\mapsto t^2$ is not operator monotone, and hence one cannot simply square both sides of \eqref{needed_power}. For this, Lin proposed an elegant method for such a process, see \cite{lin_studia, lin_jmaa}. The technique proposed in these references was then used by several authors to present powers of operator inequalities. In particular, it is shown in \cite{fu} that one can take the $p-$ power of \eqref{needed_power} as follows
\begin{equation}\label{needed_power_2}
\Phi\left(A\nabla B\right)^p\leq \left(\frac{(m+M)^2}{4^{\frac{2}{p}}mM}\right)^{p} \Phi(A\sharp B)^p, \quad p\geq 2.
\end{equation}
When $p=2$, this gives the same conclusion as Lin's.\\

In this remark, we follow a simple approach to obtain these inequalities. For this, we need to remind the reader of the following fact, see \cite[Theorem 2, p. 204]{furuta}: If $A>0$ and $mI\leq B\leq MI$, we have
\begin{equation}\label{needed_power_3}
B\leq A\Rightarrow B^p\leq \frac{\left(M^{p-1}+m^{p-1}\right)^2}{4m^{p-1}M^{p-1}}A^p,\quad p\geq 2.
\end{equation}
Now applying \eqref{needed_power_3} on \eqref{needed_power}, we obtain 
\begin{equation}\label{needed_power_4}
\Phi(A\nabla B)^p\leq \frac{\left(M^{p-1}+m^{p-1}\right)^2}{4m^{p-1}M^{p-1}}\left(\frac{M+m}{2\sqrt{mM}}\right)^p\Phi(A\sharp B)^p,\quad p\geq 2.
\end{equation}
When $p=2,$ we obtain the same squared version as Lin's inequality.  However, for other values  of $p$, it is interesting to compare \eqref{needed_power_2} with \eqref{needed_power_4}. For this end, one can define the function
$$f_p(x)=\left(\frac{(x+1)^2}{4^{2/p}x}\right)^p-\frac{\left(1+x^{p-1}\right)^2}{4x^{p-1}}\left(\frac{1+x}{2\sqrt{x}}\right)^p,\quad x\geq 1, p\geq 2.$$
Direct calculations show that $f_{2.5}(7)>0$ while $f_5(8)<0$, which means that neither \eqref{needed_power_2} nor \eqref{needed_power_4} is uniformly better than the other. However, this entails the following refinement of \eqref{needed_power_2} and \eqref{needed_power_4}:
$$\Phi(A\nabla B)^p\leq \eta\; \Phi(A\sharp B)^p,$$ where
$$\eta=\min\left\{\left(\frac{(m+M)^2}{4^{\frac{2}{p}}mM}\right)^{p},\frac{\left(M^{p-1}+m^{p-1}\right)^2}{4m^{p-1}M^{p-1}}\left(\frac{M+m}{2\sqrt{mM}}\right)^p\right\}.$$
In a similar manner we can also obtain the following  inequality:
\[\Phi {{\left( A\nabla B \right)}^{p}}\le \eta {{\left( \Phi \left( A \right)\sharp\Phi \left( B \right) \right)}^{p}}.\]
It should be noted that all inequalities in this article can be powered in the same way as above.
\end{remark}

\begin{proposition}\label{8}
	Let $A,B$ be two positive operators and $m_1,m_2,M_1,M_2$ be positive real numbers, satisfying  $0<m_2I\le A\le m_1I<M_1I\le B\le M_2I$. Then, for $0\leq v\leq 1$,
		\begin{equation}\label{1}
		\frac{m_2{{\sharp}_{v}}M_2}{m_2{{\nabla }_{v}}M_2}A{{\nabla }_{v}}B\le A{{\sharp}_{v}}B\le \frac{m_1{{\sharp}_{v}}M_1}{m_1{{\nabla }_{v}}M_1}A{{\nabla }_{v}}B,
		\end{equation}
		and
		\begin{equation}\label{2}
		\frac{m_1{{\sharp}_{v}}M_1}{m_1{{!}_{v}}M_1}A{{!}_{v}}B\le A{{\sharp}_{v}}B\le \frac{m_2{{\sharp}_{v}}M_2}{m_1{{!}_{v}}M_2}A{{!}_{v}}B.
		\end{equation}
		
\end{proposition}
\begin{proof}
	Define 
	\[{{\hat{f}}_{v}}\left( x \right)=\frac{{{x}^{v}}}{\left( 1-v \right)+vx},\quad\text{ where }\frac{M_1}{m_1}\le x\le \frac{M_2}{m_2}\text{ and }v\in \left[ 0,1 \right].\] 
	A simple calculation reveals that ${{\hat{f}}_{v}}\left( x \right)$ is monotone decreasing when $x\ge 1$ (and increasing when $0<x\le 1$). Since $\frac{M_1}{m_1}\ge 1$, ${{\hat{f}}_{v}}\left( x \right)$ attains its maximum  at $x=\frac{M_1}{m_1}$ and minimum  at $x=\frac{M_2}{m_2}$. Thus we can write
	\[{{\hat{f}}_{v}}\left( \frac{M_2}{m_2} \right)\le {{\hat{f}}_{v}}\left( x \right)\le {{\hat{f}}_{v}}\left( \frac{M_1}{m_1} \right),\]
	whenever $\frac{M_1}{m_1}\le x\le \frac{M_2}{m_2}$. By replacing $x$ with ${{A}^{-\frac{1}{2}}}B{{A}^{-\frac{1}{2}}}$ and applying a standard functional calculus argument, we get the desired inequalites in \eqref{1}.
	
	For \eqref{2}, define 
	\[{{g}_{v}}\left( x \right)=\left( 1-v \right){{x}^{v}}+v{{x}^{v-1}},\quad\text{ where }\frac{M_1}{m_1}\le x\le \frac{M_2}{m_2}\text{ and }v\in \left[ 0,1 \right].\] 
	A simple calculation shows that $g_{v}^{'}\left( x \right)=v\left( 1-v \right){{x}^{v-2}}\left( x-1 \right)$. Thus we obtain similarly
	\[{{g}_{v}}\left( \frac{M_1}{m_1} \right)\le {{g}_{v}}\left( x \right)\le {{g}_{v}}\left( \frac{M_2}{m_2} \right),\]
	whenever $\frac{M_1}{m_1}\le x\le \frac{M_2}{m_2}$. Then, an argument similar to the above implies \eqref{2}.
\end{proof}

\begin{remark} We comment on the sharpness of our results compared to some known results in the literature.
\begin{enumerate}
	\item The constants in Theorem \ref{19}, Corollary \ref{10}, and Proposition \ref{8} are best possible in the sense that smaller quantities cannot replace them. As a matter of fact, $f_v$ and $g_v$ are continuous functions on $0< s \leq x \leq t$, so that $f_v(x) \leq f_v(t)$ is a sharp inequality for example. For  the reader convenience, we will show that our results are stronger than the inequalities obtained in \cite{3} and \cite{01}.
	
	On account of \cite[Corollary 3]{01}, if $a,b>0$ and $v\in \left[ 0,1 \right]$, we have
	\[K{{\left( t \right)}^{r}}a{{\sharp}_{v}}b\le a{{\nabla }_{v}}b,\] 
	where $r=\min \left\{ v,1-v \right\}$ and $t=\frac{b}{a}$. Letting  $a=m_1$ and $b=M_1$ we get
	\begin{equation*}
	\frac{m_1{{\sharp}_{v}}M_1}{m_1{{\nabla }_{v}}M_1}\le \frac{1}{K{{\left( \frac{M_1}{m_1} \right)}^{r}}}.
	\end{equation*}
	
	In addition, Liao et al. in \cite[Corollary 2.2]{3} proved that
	\[a{{\nabla }_{v}}b\le K{{\left( t \right)}^{R}}a{{\sharp}_{v}}b,\] 
	where $R=\max \left\{ v,1-v \right\}$. By choosing $a=m_2$ and $b=M_2,$ we have
	\begin{equation*}
	\frac{m_2{{\nabla }_{v}}M_2}{m_2{{\sharp}_{v}}M_2}\le K{{\left( \frac{M_2}{m_2} \right)}^{R}}.
	\end{equation*}
	\item 	The assumption on $A$ and $B$ ( $sA \leq B \le tA$ in Theorem \ref{19}) is more general  than $mI \le A,B \le MI$ in Corollary \ref{10} and the conditions (i) or (ii) in Lemma \ref{8}.
	The conditions (i) or (ii) in Lemma \ref{8} imply $I \le \frac{M_1}{m_1}I\le A^{-1/2}BA^{-1/2} \le \frac{M_2}{m_2}I$ which is a special case of $0< sI \leq A^{-1/2}BA^{-1/2} \le tI$ with $s=\frac{M_1}{m_1}$ and $t=\frac{M_2}{m_2}$.
\end{enumerate}
\end{remark}

An application of Proposition \ref{8} under positive linear maps is given  as follows.
\begin{corollary}\label{corollary2_2}
Let $A,B$ be two positive operators, $\Phi$ be a unital positive linear map, $v\in \left[ 0,1 \right]$ and $m_1,m_2,M_1,M_2$ be positive real numbers.
\begin{itemize}
	\item[(i)] If $0<m_2I\leq A \leq m_1I \leq M_1I \leq B \leq M_2I$, then
	\begin{equation}\label{remark1_3_ineq01}
	\Phi(A)\sharp_v\Phi(B) \leq \frac{m_1\sharp_vM_1}{m_2\sharp_v M_2}\frac{m_2\nabla_vM_2}{m_1\nabla_vM_1}\Phi(A\sharp_vB).
	\end{equation}
	\item[(ii)] If $0<m_2I\leq B \leq m_1I \leq M_1I \leq A \leq M_2I$, then
	\begin{equation}\label{remark1_3_ineq02}
	\Phi(A)\sharp_v\Phi(B) \leq \frac{M_1\sharp_vm_1}{M_2\sharp_v m_2}\frac{M_2\nabla_vm_2}{M_1\nabla_vm_1}\Phi(A\sharp_vB).
	\end{equation}
\end{itemize}
\end{corollary}

\begin{proof}
We prove (ii). As we noted in the proof of Proposition \ref{8}, $\hat{f}_v(x)$ is increasing on $0<x \leq 1$,
we have
$$
\hat{f}_v\left(\frac{m_2}{M_2}\right) \le \hat{f}_v(x) \le \hat{f}_v\left(\frac{m_1}{M_1}\right)
$$ 
whenever $\frac{m_2}{M_2} \leq x \leq \frac{m_1}{M_1} (\leq 1)$.
Thus we have
$$
\frac{M_2\sharp_vm_2}{M_2\nabla_vm_2}A\nabla_vB \leq A \sharp_vB \leq \frac{M_1\sharp_vm_1}{M_1\nabla_vm_1}A\nabla_vB.
$$
Replacing $A$ and $B$ by $\Phi(A)$ and $\Phi(B)$ in the second inequality above and taking $\Phi$ of both sides in the first inequality above, we obtain
$$
\Phi(A)\sharp_v\Phi(B) \leq \frac{M_1\sharp_vm_1}{M_1\nabla_vm_1}\Phi(A\nabla_vB),\,\,\,\, \frac{M_2\sharp_vm_2}{M_2\nabla_vm_2}\Phi(A\nabla_vB) \leq \Phi(A \sharp_vB). 
$$
Combining these inequalities, we have the desired result.
\end{proof}

\begin{remark}\label{remark1_3}
In the special case when $v=\frac{1}{2}$, our inequalities in Corollary \ref{corollary2_2} improve inequality \eqref{24}. This follows from the fact that $\frac{{{m}_{1}}\sharp{{M}_{1}}}{{{m}_{1}}\nabla {{M}_{1}}},\frac{{{M}_{1}}\sharp{{m}_{1}}}{{{M}_{1}}\nabla {{m}_{1}}}\le 1$.
\end{remark}

As an application of Proposition \ref{8}, we estimate the bounds of Tsallis relative operator entropy defined by \cite{furuichi}:
$$
T_v(A|B) \equiv A^{1/2}\ln_v\left(A^{-1/2}BA^{-1/2}\right)A^{1/2}=\frac{A\sharp_vB-A}{v}
$$
for two positive operators $A,B$ and $v\in \left( 0,1 \right]$, where $\ln_vx\equiv \frac{x^v -1}{v}$ is defined on $x>0$ for $v\neq 0$.
\begin{corollary}\label{c8}
	Let $A,B$ be two positive operators, $v\in \left( 0,1 \right]$ and $m_1,m_2,M_1,M_2$ be positive real numbers.
	\begin{itemize}
		\item[(i)] If $0<m_2I\le A\le m_1I<M_1I\le B\le M_2I$, then
		$$
		\frac{\left( m_2\sharp_vM_2\right) A\nabla_v B -\left(m_2\nabla_v M_2\right)A}{v\left(m_2\nabla_v M_2\right)}\le T_v(A|B) \le \frac{\left( m_1\sharp_vM_1\right) A\nabla_v B -\left(m_1\nabla_v M_1\right)A}{v\left(m_1\nabla_v M_1\right)}
		$$
		and
		$$
		\frac{\left( m_1\sharp_vM_1\right) A\nabla_v B -\left(m_1!_v M_1\right)A}{v\left(m_1!_v M_1\right)}\le T_v(A|B) \le \frac{\left( m_2\sharp_vM_2\right) A\nabla_v B -\left(m_2!_v M_2\right)A}{v\left(m_2!_v M_2\right)}.
		$$
		\item[(ii)] If $0<m_2I\le B\le m_1I<M_1I\le A\le M_2I$, then
		$$
		\frac{\left( M_2\sharp_vm_2\right) A\nabla_v B -\left(M_2\nabla_v m_2\right)A}{v\left(M_2\nabla_v m_2\right)}\le T_v(A|B) \le \frac{\left( M_1\sharp_vm_1\right) A\nabla_v B -\left(M_1\nabla_v m_1\right)A}{v\left(M_1\nabla_v m_1\right)}
		$$
		and
		$$
		\frac{\left( M_1\sharp_vm_1\right) A\nabla_v B -\left(M_1!_v m_1\right)A}{v\left(M_1!_v m_1\right)}\le T_v(A|B) \le \frac{\left( M_2\sharp_vm_2\right) A\nabla_v B -\left(M_2!_v m_2\right)A}{v\left(M_2!_v m_2\right)}.
		$$
	\end{itemize}
\end{corollary}

In the following, we present related results for $v\notin \left[ 0,1 \right]$.
It is known that if $a,b>0$ and $v\notin \left[ 0,1 \right]$, then $a{{\nabla }_{v}}b\le a{{\sharp}_{v}}b$ (this  fact  has  been  studied in some details  in \cite{2} and was refined later in \cite{sab_mia}), which implies
\begin{equation}\label{5}
A{{\nabla }_{v}}B\le A{{\sharp}_{v}}B,
\end{equation}
whenever $A,B\in \mathcal{B}\left( \mathcal{H} \right)$ are two positive  operators.

The following result provides a multiplicative refinement and reverse of inequality \eqref{5}. We omit the details of the proof since it is similar to the proof of Proposition \ref{8}. We also remark that in \cite{2,sab_mia} only additive refinements were given. Here we present multiplicative refinements and reverses.
\begin{proposition}\label{13}
	Let $A,B$ be two positive operators and $m_1,m_2,M_1,M_2$ be positive real numbers.
	 If $0<m_2I\le A\le m_1I<M_1I\le B\le M_2I$, then
		\[\frac{m_1{{\sharp}_{v}}M_1}{m_1{{\nabla }_{v}}M_1}A{{\nabla }_{v}}B\le A{{\sharp}_{v}}B\le \frac{m_2{{\sharp}_{v}}M_2}{m_2{{\nabla }_{v}}M_2}A{{\nabla }_{v}}B\quad\text{ for }v>1\]
		and
		\[\frac{m_1{{!}_{v}}M_1}{m_1{{\sharp}_{v}}M_1}A{{\sharp}_{v}}B\le A{{!}_{v}}B\le \frac{m_2{{!}_{v}}M_2}{m_2{{\sharp}_{v}}M_2}A{{\sharp}_{v}}B\quad\text{ for }v<0.\]

\end{proposition}

\subsection{Related inequalities for operator monotone functions}

In this section, we present operator  inequalities involving positive linear maps and operator monotone functions. We begin with the following application of Theorem \ref{19}.

\begin{theorem}\label{theorem_c}
Let $\Phi $ be a positive linear map, $A,B$ be  positive operators such that $sA\le B\le tA$ for some scalars $0<s\le t$, and let ${{!}_{v}}\le {{\sigma }_{v}},{{\tau }_{v}}\le {{\nabla }_{v}}$ for any $v\in \left[ 0,1 \right]$.

If $f$ is a nonzero operator monotone function on $\left[ 0,\infty  \right)$, then
\begin{equation} \label{theorem_b_ineq01}
f\left( \Phi \left( A \right) \right){{\tau }_{v}}f\left( \Phi \left( B \right) \right)\le \xi \psi f\left( \Phi \left( A{{\sigma }_{v}}B \right) \right).
\end{equation}
If $g$ is a nonzero operator monotone decreasing function on $\left[ 0,\infty  \right)$, then
\[g\left( \Phi \left( A{{\sigma }_{v}}B \right) \right)\le \xi \psi \left( g\left( \Phi \left( A \right) \right){{\tau }_{v}}g\left( \Phi \left( B \right) \right) \right),\]
where $\xi =\max \left\{ \frac{\left( 1-v \right)+vs}{{{s}^{v}}},\frac{\left( 1-v \right)+vt}{{{t}^{v}}} \right\}$ and $\psi =\max \left\{ {{s}^{v}} \left(\left( 1-v \right)+\frac{v}{s}\right),{{t}^{v}} \left( \left( 1-v \right)+\frac{v}{t}\right) \right\}$.
\end{theorem}
\begin{proof}
On account of Theorem \ref{19}, we have
\begin{equation}\label{9}
\Phi \left( A \right){{\nabla }_{v}}\Phi \left( B \right)\le \xi \psi \Phi \left( A{{\sigma }_{v}}B \right).
\end{equation}
It follows from the inequality \eqref{9} that
\[\begin{aligned}
 f\left( \Phi \left( A \right) \right){{\tau }_{v}}f\left( \Phi \left( B \right) \right)&\le f\left( \Phi \left( A \right) \right){{\nabla }_{v}}f\left( \Phi \left( B \right) \right) \\ 
& \le f\left( \Phi \left( A \right){{\nabla }_{v}}\Phi \left( B \right) \right) \\ 
& \le f\left( \xi \psi \Phi \left( A{{\sigma }_{v}}B \right) \right)\\ 
& \le \xi \psi f\left( \Phi \left( A{{\sigma }_{v}}B \right) \right),
\end{aligned}\]
where, in the last line, we have used the fact that for $\alpha\geq 1,$ $f(\alpha A)\leq \alpha f(A)$ when $f$ is operator monotone.\\
For the second inequality we can write
\[g\left( \Phi \left( A \right) \right){{\tau }_{v}}g\left( \Phi \left( B \right) \right)\ge g\left( \Phi \left( A \right){{\nabla }_{v}}\Phi \left( B \right) \right)\ge g\left( \xi \psi \Phi \left( A{{\sigma }_{v}}B \right) \right)\ge \frac{1}{\xi \psi }g\left( \Phi \left( A{{\sigma }_{v}}B \right) \right),\]
where the first inequality follows from \cite[Remark 2.6]{6}.
\end{proof}

\begin{remark}
Taking $v=\frac{1}{2}$, $s=\frac{m}{M}$ and $t=\frac{M}{m}$ in Theorem \ref{theorem_c}, we have $\xi=\psi=\frac{1}{2}\left(\sqrt{\frac{m}{M}}+\sqrt{\frac{M}{m}}\right)$. Since $\xi \psi=\frac{(M+m)^2}{4Mm}$, the inequality (\ref{theorem_b_ineq01}) recovers the inequality (\ref{Hoa_ineq01}).
\end{remark}

\begin{remark}
In Theorem \ref{theorem_c}, it is proved that for two means $\tau_v,\sigma_v,$  we have:
If $f$ is a nonzero operator monotone function on $\left[ 0,\infty  \right)$, then
\begin{equation} \label{theorem_c_ineq01}
f\left( \Phi \left( A \right) \right){{\tau }_{v}}f\left( \Phi \left( B \right) \right)\le \xi \psi f\left( \Phi \left( A{{\sigma }_{v}}B \right) \right).
\end{equation}

We can modify the constant $\xi\psi$ as follows. Take the function $h_v(x)=(1-v+vx)(1-v+\frac{v}{x}), 0\leq v\leq 1, s\leq x\leq t.$ Direct computations show that
$$h_v'(x)=\frac{v(1-v)(x^2-1)}{x^2},$$ which implies 
$$h_v(x)\leq\max\{h_v(s),h_v(t)\}:=\alpha.$$

So, if $!_v\leq\tau_v,\sigma_v\leq\nabla_v$, then we obtain by the similar way to the proof of Theorem \ref{theorem_c},

\begin{equation} 
f\left( \Phi \left( A \right) \right){{\tau }_{v}}f\left( \Phi \left( B \right) \right)\le \alpha f\left( \Phi \left( A{{\sigma }_{v}}B \right) \right).
\end{equation}

That is, the constant $\xi\psi$ has been replaced by $\alpha.$

Notice that $\xi\psi=\alpha$ in case both maxima (for $\xi,\psi$) are attained at the same $t$ or $s$. If $s,t\leq 1$ or $s,t\geq 1$, we do have $\alpha=\xi\psi$. But, if $s<1$ and $t>1$,  it can be seen that that $\alpha \le  \xi\psi$, which is a better approximation.
\end{remark}

Notice that Theorem \ref{theorem_c} is a multiplicative inequality, where the two sides of the given inequalities are related via scalar multiplication. The next result is an additive version, where upper bounds of the difference between $f\left( \Phi \left( A \right) \right){{\tau }}f\left( \Phi \left( B \right) \right)$ and $f\left( \Phi \left( A{{\sigma }}B \right) \right)$ are given.

\begin{corollary}
Let $\Phi $ be a positive unital linear map, $A,B$ be two  positive  operators such that $mI\le A,B\le MI$ for some scalars $0<m<M$, and let ${{!}}\le {{\sigma }},{{\tau }}\le {{\nabla }}$.
	
If $f$ is a nonzero operator monotone function on $\left[ 0,\infty  \right)$, then
\[f\left( \Phi \left( A \right) \right){{\tau }}f\left( \Phi \left( B \right) \right)-f\left( \Phi \left( A{{\sigma }}B \right) \right)\le \frac{{{\left( M-m \right)}^{2}}}{4Mm}f\left( M \right)I.\]
Further, if $g$ is a nonzero operator monotone decreasing function on $\left[ 0,\infty  \right)$, then
\[g\left( \Phi \left( A{{\sigma }}B \right) \right)-g\left( \Phi \left( A \right) \right){{\tau }}g\left( \Phi \left( B \right) \right)\le \frac{{{\left( M-m \right)}^{2}}}{4Mm}g\left( m \right)I.\]
\end{corollary}
\begin{proof}
Since $mI\le A,B\le MI$, we get $mI\le \Phi \left( A{{\sigma }}B \right)\le MI$. Therefore,
\[f\left( m \right)I\le f\left( \Phi \left( A{{\sigma }}B \right) \right)\le f\left( M \right)I,\]
and hence, 
\[\begin{aligned}
 f\left( \Phi \left( A \right) \right){{\tau }}f\left( \Phi \left( B \right) \right)-f\left( \Phi \left( A{{\sigma }}B \right) \right)&\le \left( \frac{{{\left( M+m \right)}^{2}}}{4Mm}-1 \right)f\left( \Phi \left( A{{\sigma }_{v}}B \right) \right) \\ 
& \le \frac{{{\left( M-m \right)}^{2}}}{4Mm}f\left( M \right)I.  
\end{aligned}\] 
If $g$ is operator monotone decreasing, then $g\left( M \right)I\le g\left( \Phi \left( A \right) \right),g\left( \Phi \left( B \right) \right)\le g\left( m \right)I$. Therefore 
\[g\left( M \right)I\le g\left( \Phi \left( A \right) \right){{\tau }}g\left( \Phi \left( B \right) \right)\le g\left( m \right)I.\]
By repeating the same argument as above we get the desired result.
\end{proof}
It is shown in \cite{ghaemi} that if $f$ is an operator monotone function on $\left[ 0,\infty  \right)$ and $sA\le B\le tA$, for some scalars $0<s\le t$, then
\begin{equation}\label{15}
f\left( A \right){{\sharp}_{v}}f\left( B \right)\le \max \left\{ S\left( s \right),S\left( t \right) \right\}f\left( A{{\sharp}_{v}}B \right),
\end{equation}
where $S\left( t \right)=\frac{{{t}^{\frac{1}{t-1}}}}{e\log {{t}^{\frac{1}{t-1}}}}$. As a byproduct of Theorem \ref{theorem_c}, we have the following generalization of \eqref{15}:
	\[f\left( A \right){{\tau }_{v}}f\left( B \right)\le \xi \psi f\left( A{{\sigma }_{v}}B \right).\] 
However, we can improve \eqref{15} as follows, where one can show that the constant $\xi$ below is smaller than $ \max \left\{ S\left( s \right),S\left( t \right) \right\}$ from \eqref{15}.
\begin{corollary}
Let $f:\left[ 0,\infty  \right)\to \left[ 0,\infty  \right)$ be an operator monotone function, $A,B$ be positive operators such that $sA\le B\le tA$, for some scalars $0<s\le t$. Then for all $v\in \left[ 0,1 \right],$
\[f\left( A \right){{\sharp}_{v}}f\left( B \right)\le \xi f\left( A{{\sharp}_{v}}B \right),\]
where $\xi =\max \left\{ \frac{\left( 1-v \right)+vs}{{{s}^{v}}},\frac{\left( 1-v \right)+vt}{{{t}^{v}}} \right\}$.
\end{corollary}

\begin{corollary}
	Let $A,B$ be as in Theorem \ref{19} and let $g$ be an operator monotone decreasing function. If $\sigma_v$ is a symmetric mean between $\nabla_v$ and $!_v,$ $0\leq v\leq 1$, then for any vector $h\in\mathcal{H}$,
\[\left\langle g\left( A{{\sigma }_{v}}B \right)h,h \right\rangle \le \xi \psi {{\left\langle g\left( A \right)h,h \right\rangle }^{1-v}}{{\left\langle g\left( B \right)h,h \right\rangle }^{v}},\]
for the same $\xi,\psi$ from above.
\end{corollary}
\begin{proof}
	Notice that, for $0\leq v\leq 1,$
\[\begin{aligned}
 \left\langle g\left( A{{\sigma }_{v}}B \right)h,h \right\rangle &\le \left\langle g\left( A{{!}_{v}}B \right)h,h \right\rangle  \le \psi \left\langle g\left( \psi A{{!}_{v}}B \right)h,h \right\rangle  \\ 
& \le \psi \left\langle g\left( A{{\sharp}_{v}}B \right)h,h \right\rangle   \le \xi \psi \left\langle g\left( \xi A{{\sharp}_{v}}B \right)h,h \right\rangle  \\
&\le \xi \psi \left\langle g\left( A{{\nabla }_{v}}B \right)h,h \right\rangle \le \xi \psi \left\langle \left( g\left( A \right){{\sharp}_{v}}g\left( B \right) \right)h,h \right\rangle  \\ 
& \le \xi \psi {{\left\langle g\left( A \right)h,h \right\rangle }^{1-v}}{{\left\langle g\left( B \right)h,h \right\rangle }^{v}}, 
\end{aligned}\]
where we have used \cite[Lemma 8]{bourin} to obtain the last inequality.
\end{proof}

The following lemma, which we need in our analysis, can be found in \cite[Lemma 3.11]{sab_laaa}.
\begin{lemma}\label{14}
	Let $f:\mathbb{R}\to \mathbb{R}$ be operator convex, $A,B\in \mathcal{B}\left( \mathcal{H} \right)$ be two self-adjoint operators and let $v\notin \left[ 0,1 \right]$. Then
	\begin{equation}\label{11}
	f\left( A \right){{\nabla }_{v}}f\left( B \right)\le f\left( A{{\nabla }_{v}}B \right),
	\end{equation}
	and the reverse inequality holds if $f$ is operator concave.
\end{lemma}
\begin{proposition}
	Let $A,B$ be two positive operators and $m_1,m_2,M_1,M_2$ be positive real numbers.
	\begin{itemize}
		\item[(i)] If $0<m_2I\le A\le m_1I<M_1I\le B\le M_2I$, then
		\begin{equation}\label{20}
	f\left( A{{\sharp}_{v}}B \right)\le \frac{m_2{{\sharp}_{v}}M_2}{m_2{{\nabla }_{v}}M_2}\left( f\left( A \right){{\sharp}_{v}}f\left( B \right) \right)\quad\text{ for }v>1
		\end{equation}
		for any operator monotone function $f$, and
		\begin{equation}\label{22}
g\left( A \right){{\sharp}_{v}}g\left( B \right)\le \frac{m_2{{\sharp}_{v}}M_2}{m_2{{\nabla }_{v}}M_2}g\left( A{{\sharp}_{v}}B \right)\quad\text{ for }v>1
		\end{equation}
		for any operator monotone decreasing function $g$.
		\item[(ii)] If $0<m_2I\le B\le m_1I<M_1I\le A\le M_2I$, then
		\[f\left( A{{\sharp}_{v}}B \right)\le \frac{M_2{{\sharp}_{v}}m_2}{M_2{{\nabla }_{v}}m_2}\left( f\left( A \right){{\sharp}_{v}}f\left( B \right) \right)\quad\text{ for }v<0\]
		for any operator monotone function $f$, and
		\[g\left( A \right){{\sharp}_{v}}g\left( B \right)\le \frac{M_2{{\sharp}_{v}}m_2}{M_2{{\nabla }_{v}}m_2}g\left( A{{\sharp}_{v}}B \right)\quad\text{ for }v<0\]
		for any operator monotone decreasing function $g$.
	\end{itemize}
\end{proposition}
\begin{proof}
Let  us  prove  (i).  It follows from Proposition \ref{13} and Lemma \ref{14} that, for $v>1$,
	\[\begin{aligned}
	f\left( A{{\sharp}_{v}}B \right)&\le f\left( \frac{m_2{{\sharp}_{v}}M_2}{m_2{{\nabla }_{v}}M_2}A{{\nabla }_{v}}B \right)  \le \frac{m_2{{\sharp}_{v}}M_2}{m_2{{\nabla }_{v}}M_2}f\left( A{{\nabla }_{v}}B \right) \\ 
	& \le \frac{m_2{{\sharp}_{v}}M_2}{m_2{{\nabla }_{v}}M_2}\left( f\left( A \right){{\nabla }_{v}}f\left( B \right) \right)  \le \frac{m_2{{\sharp}_{v}}M_2}{m_2{{\nabla }_{v}}M_2}\left( f\left( A \right){{\sharp}_{v}}f\left( B \right) \right)  
	\end{aligned}\]
which gives the desired inequality \eqref{20}.	

Since $g$ is operator monotone decreasing,  $\frac{1}{g}$ is operator monotone. Now by applying the inequality \eqref{20} for $f=\frac{1}{g}$, we get
\begin{equation}\label{23}
{{g}^{-1}}\left( A{{\sharp}_{v}}B \right)\le \frac{m_2{{\sharp}_{v}}M_2}{m_2{{\nabla }_{v}}M_2}\left( {{g}^{-1}}\left( A \right){{\sharp}_{v}}{{g}^{-1}}\left( B \right) \right).
\end{equation}
Consequently, from \eqref{23} we get
\[g\left( A{{\sharp}_{v}}B \right)\ge \frac{m_2{{\nabla }_{v}}M_2}{m_2{{\sharp}_{v}}M_2}{{\left( {{g}^{-1}}\left( A \right){{\sharp}_{v}}{{g}^{-1}}\left( B \right) \right)}^{-1}}=\frac{m_2{{\nabla }_{v}}M_2}{m_2{{\sharp}_{v}}M_2}\left( g\left( A \right){{\sharp}_{v}}g\left( B \right) \right),\]
which is  just \eqref{22}. 
\end{proof}

\section*{Acknowledgement}
The author (S.F.) was partially supported by JSPS KAKENHI Grant Number 16K05257. 
The corresponding author (M. S.) was supported by a sabbatical leave from Princess Sumaya University for Technology.

\vskip 0.3 true cm

{\tiny (S. Furuichi) Department of Information Science, College of Humanities and Sciences, Nihon University, 3-25-40, Sakurajyousui, Setagaya-ku, Tokyo, 156-8550, Japan.}

{\tiny \textit{E-mail address:} furuichi@chs.nihon-u.ac.jp}

	{\tiny \vskip 0.3 true cm }
	
	{\tiny (H.R. Moradi) Young Researchers and Elite Club, Mashhad Branch, Islamic Azad
		University, Mashhad, Iran. }
	
	{\tiny \textit{E-mail address:} hrmoradi@mshdiau.ac.ir }
	
{\tiny \vskip 0.3 true cm }

{\tiny (M. Sababheh) Department of Mathematics, University of Sharjah, Sharjah UAE, 
	\textit{Email adress:} m.sababheh@sharjah.ac.ae }

{\tiny (M. Sababheh) Dept. of Basic Sciences, Princess Sumaya Univ. for Tech., Amman,
	Jordan. \textit{E-mail address:} sababheh@psut.edu.jo 
\end{document}